\documentclass[12pt,dvipsnames]{amsart}
\usepackage[colorlinks=true, pdfstartview=FitV, linkcolor=blue, citecolor=blue, urlcolor=blue]{hyperref}
\usepackage{geometry,pb-diagram}
\usepackage{amssymb, amscd, amsmath, amsfonts, amsthm}
\usepackage{stmaryrd}
\usepackage{verbatim}
\usepackage{multicol}
\newcommand{\idiot}[1]{\vspace{5 mm}\par \noindent
\marginpar{\textsc{For longer version}}
\framebox{\begin{minipage}[c]{.99 \textwidth}
#1 \end{minipage}}\vspace{5 mm}\par}

\newcommand{\todo}[1]{\vspace{5 mm}\par \noindent
\marginpar{\textsc{ToDo}}
\framebox{\begin{minipage}[c]{.99 \textwidth}
\tt #1 \end{minipage}}\vspace{5 mm}\par}

\renewcommand{\todo}[1]{}
\renewcommand{\idiot}[1]{}
\def\unprotectedboldentry#1{\textcolor{Red}{\textbf{#1}}}
\def\boldentry{\protect\unprotectedboldentry}
\usepackage{tikz}
\usetikzlibrary{calc}
\usepackage{ifthen}
\newcommand{\tikztableau}[2][scale=0.6,every node/.style={font=\small}]{
    \def\newtableau{#2}
    \begin{array}{c}
    \begin{tikzpicture}[#1]
    \coordinate (x) at (-0.5,0.5);
    \coordinate (y) at (-0.5,0.5);
    \foreach \row in \newtableau {
        \coordinate (x) at ($(x)-(0,1)$);
        \coordinate (y) at (x);
        \foreach \entry in \row {
            \ifthenelse{\equal{\entry}{X}}
               {
                \node (y) at ($(y) + (1,0)$) {};
                \fill[color=gray!10] ($(y)-(0.5,0.5)$) rectangle +(1,1);
                \draw[color=gray] ($(y)-(0.5,0.5)$) rectangle +(1,1);
               }
               {
                \ifthenelse{\equal{\entry}{\boldentry X}}
                   {
                    \node (y) at ($(y) + (1,0)$) {};
                    \fill[color=gray] ($(y)-(0.5,0.5)$) rectangle +(1,1);
                    \draw ($(y)-(0.5,0.5)$) rectangle +(1,1);
                   }
                   {
                    \node (y) at ($(y) + (1,0)$) {\entry};
                    \draw ($(y)-(0.5,0.5)$) rectangle +(1,1);
                   }
               }
            }
        }
    \end{tikzpicture}
    \end{array}}
\newcommand{\tikztableausmall}[1]{\tikztableau[scale=0.45,every node/.style={font=\rm\small}]{#1}}

\def\sym{\operatorname{\mathsf{Sym}}}

\def\Qsym{\operatorname{\mathsf{QSym}}}

\def \fS{{\mathfrak S}}
\def \HH{{H}}
\def\Nsym{\operatorname{\mathsf{NSym}}}

\def\SS{\hat{S}}
\def\sort{\operatorname{sort}}

\newtheorem{Theorem}{Theorem}[section]
\newtheorem{Proposition}[Theorem]{Proposition}

\newtheorem{Lemma}[Theorem]{Lemma}
\newtheorem{Example}[Theorem]{Example}
\theoremstyle{definition}
\newtheorem{Remark}[Theorem]{Remark}
\newtheorem{Definition}[Theorem]{Definition}

\begin{document}

\title[Indecomposable modules for the dual immaculate basis]{Indecomposable modules for the dual immaculate basis of quasi-symmetric functions}

\author[C. Berg \and N. Bergeron \and F. Saliola \and L. Serrano \and M. Zabrocki]{Chris Berg$^2$ 
\and Nantel Bergeron$^{1,3}$ \and Franco Saliola$^2$ \and Luis Serrano$^2$ \and Mike Zabrocki$^{1,3}$}
\address[1]{Fields Institute\\ Toronto, ON, Canada}
\address[2]{Universit\'e du Qu\'ebec \`a Montr\'eal, Montr\'eal, QC, Canada}
\address[3]{York University\\ Toronto, ON, Canada}
\date{\today}

\begin{abstract}
We construct indecomposable modules for the $0$-Hecke algebra whose characteristics are the dual immaculate basis of the quasi-symmetric functions.
\end{abstract}

\maketitle
\setcounter{tocdepth}{3}

\section{Introduction}

The algebra of symmetric functions $\sym$ has an important basis formed by Schur functions, which appear throughout mathematics. For example, as the representatives for the Schubert 
classes in the cohomology of the Grassmannian, as the characters for the irreducible representations of the symmetric group and the general linear group, or as an orthonormal basis for the space of symmetric functions,
to name a few. The algebra $\Nsym$ of noncommutative symmetric functions projects under the forgetful map onto $\sym$, which injects into the algebra $\Qsym$ of quasi-symmetric functions. $\Nsym$ and $\Qsym$ are dual Hopf algebras.

In \cite{BBSSZ}, the authors developed a basis for $\Nsym$, which satisfied many of the combinatorial properties of Schur functions. This basis, called the \textit{immaculate basis} $\{\fS_\alpha\}$, projects onto Schur functions under the forgetful map. When 
indexed by a partition, the corresponding projection of the immaculate function is precisely the Schur function of the given partition.

The dual basis $\{\fS_\alpha^*\}$ is a basis for $\Qsym$. 
%It is a classical result that the Schur functions represent the characters of the irreducible representations of the symmetric group.  
The main goal of this paper is to express the dual immaculate functions as characters of a representation, in the same way that Schur functions are the characters of the irreducible representations of the symmetric group. We achieve this in Theorem \ref{thm:repthry}, where we realize them as the characteristic of certain indecomposable representations of the \textit{0-Hecke algebra}.
% The main result of this paper (Theorem \ref{thm:repthry}) is to realize the dual basis as the 
%characteristic of certain indecomposable representations of the \textit{0-Hecke algebra}.

\subsection{Acknowledgements}
This work is supported in part by  NSERC.
It is partially the result of a working session at the Algebraic
Combinatorics Seminar at the Fields Institute with the active
participation of C. Benedetti, J. S\'anchez-Ortega, O. Yacobi, E. Ens,  H. Heglin, D. Mazur and T. MacHenry.

This research was facilitated by computer exploration using the open-source
mathematical software \texttt{Sage}~\cite{sage} and its algebraic
combinatorics features developed by the \texttt{Sage-Combinat}
community~\cite{sage-co}.
\section{Prerequisites}
\subsection{The symmetric group}

The symmetric group $S_n$ is the group generated by the set of $\{ s_1, s_2, \ldots, s_{n-1}\}$
satisfying the following relations:

\vspace{-.23in}
\begin{align*}
s_i^2 &= 1;\\
 s_i s_{i+1}s_i &= s_{i+1}s_i s_{i+1};\\
s_i s_j &= s_j s_i \textrm{ if } |i-j| > 1.
\end{align*}

\subsection{Compositions and combinatorics} \label{sec:compositions}

A \textit{partition} of a non-negative integer $n$ is a tuple
$\lambda = [\lambda_1, \lambda_2, \dots, \lambda_m]$ of positive integers satisfying 
$\lambda_1 \geq \lambda_2 \geq \cdots \geq \lambda_m$ which sum to $n$. If $\lambda$ is a partition of $n$, one writes 
$\lambda \vdash n$. (When needed, we will consider partitions with zeroes at the end, but they are equivalent to the underlying partition made of positive numbers.) Partitions are of particular importance in 
algebraic combinatorics, as they index a basis 
for the symmetric functions of degree $n$, $\sym_n$, and the character ring 
for the representations of the symmetric group $S_n$, among others. These concepts are intimately 
connected; we assume the reader is well versed in this area (see for instance \cite{Sagan} for background details).

A \textit{composition} of a non-negative integer $n$ is a tuple 
$\alpha = [\alpha_1, \alpha_2, \dots, \alpha_m]$ of positive 
integers which sum to $n$. If $\alpha$ is a composition of $n$, one writes $\alpha \models n$.
%We let $\C_n$ denote the set of all compositions of $n$, and 
%$\C$ denote the set of all compositions. % do we use this?
The entries $\alpha_i$ of the composition are referred to as the parts
of the composition.  The size of the composition is the sum of the parts
and will be denoted $|\alpha|$.  The length of the composition is the
number of parts and will be denoted $\ell(\alpha)$. Note that $|\alpha|=n$ and $\ell(\alpha)=m$.

Compositions of $n$ are in bijection with subsets of $\{1, 2, \dots, n-1\}$. 
We will follow the convention of identifying $\alpha = [\alpha_1, \alpha_2, \dots, \alpha_m]$ with the subset ${\mathcal S}(\alpha) = 
\{\alpha_1, \alpha_1+\alpha_2, \alpha_1+\alpha_2 + \alpha_3, \dots, \alpha_1+\alpha_2+\dots + \alpha_{m-1} \}$. 

%We let $\chi$ denote the grading preserving map 
%$\chi: \C \rightarrow \P$ which sorts a composition into weakly decreasing entries. 
If $\alpha$ and
$\beta$ are both compositions of $n$,  say that $\alpha \leq  \beta$ in refinement order if 
${\mathcal S}(\beta) \subseteq {\mathcal S}(\alpha)$. For instance, $[1,1,2,1,3,2,1,4,2] \leq [4,4,2,7]$, 
since ${\mathcal S}([1,1,2,1,3,2,1,4,2]) = \{1,2,4,5,8,10,11,15\}$ and ${\mathcal S}([4,4,2,7]) = \{4,8,10\}$.

In this presentation, compositions will be represented as diagrams of left adjusted rows of cells.
%The combinatorics of the elements that we introduce will lead us to represent our diagrams in this
%way rather than as a ribbon (as is the usual method for representing compositions when working with
%the ribbon Schur basis of $\Nsym$ and fundamental bases of $\Qsym$).
We will also use the matrix convention (`English' notation)
that the first row of the diagram is at the top and the last row is at the bottom.  For example, the composition
$[4,1,3,1,6,2]$ is represented as
\[ \tikztableausmall{{X, X, X, X},{X}, {X, X, X}, {X}, {X,X,X,X,X,X}, {X, X}}~.
\]

\subsection{Symmetric functions}
We let $\sym$ denote the ring of symmetric functions. As an algebra, $\sym$ is the ring over 
$\mathbb{Q}$ freely generated by commutative elements $\{h_1, h_2, \dots\}$. $\sym$ has a grading, defined by giving 
$h_i$ degree $i$ and extending multiplicatively. A natural basis for the degree $n$ component of $\sym$ are the complete 
homogeneous symmetric functions of degree 
$n$, $\{ h_\lambda := h_{\lambda_1} h_{\lambda_2} \cdots h_{\lambda_m} : \lambda \vdash n\}$. 
$\sym$ can be realized as the ring of invariants of the ring of power series of bounded degree 
$\mathbb{Q}[\![x_1, x_2, \dots]\!]$ in commuting variables $\{x_1, x_2, \dots\}$. Under this identification, 
$h_i$ denotes the sum of all monomials in the $x$ variables of degree $i$.

\subsection{Non-commutative symmetric functions} 

$\Nsym$ is a non-commutative analogue of $\sym$, the algebra of symmetric functions, that arises by
considering an algebra with one non-commutative generator at each positive
degree.  In addition to the relationship with the symmetric functions,
this  algebra  has  links  to  Solomon's  descent  algebra  in  type  $A$  \cite{MR},
the  algebra  of  quasi-symmetric  functions  \cite{MR},  and representation theory
of  the  type  $A$  Hecke  algebra  at  $q=0$  \cite{KT}.  It is an example
of a combinatorial Hopf algebra \cite{ABS}.  While we will follow the foundational
results  and  definitions  from  references  such  as  \cite{GKLLRT,MR},  we  have  chosen
to use notation here which is suggestive of analogous results in $\sym$.

We consider $\Nsym$ as the algebra with generators $\{\HH_1, \HH_2, \dots \}$ and 
no relations. Each generator $H_i$ is defined to be of degree $i$, 
giving $\Nsym$ the structure of a graded algebra. We let $\Nsym_n$ denote the 
graded component of $\Nsym$ of degree $n$. A basis for $\Nsym_n$ are the 
\textit{complete homogeneous functions} 
$\{\HH_\alpha := \HH_{\alpha_1} \HH_{\alpha_2} \cdots \HH_{\alpha_m}\}_{\alpha \vDash n}$ 
indexed by compositions of $n$.

\subsection{Immaculate tableaux}

\begin{Definition}
Let $\alpha$ and $\beta$ be compositions. An \emph{immaculate tableau} of shape
$\alpha$ and content $\beta$ is a labelling of the boxes of the diagram of
$\alpha$ by positive integers in such a way that:
\begin{enumerate}
\item the number of boxes labelled by $i$ is $\beta_i$;
\item the sequence of entries in each row, from left to right, is weakly increasing;
\item the sequence of entries in the \emph{first} column, from top to bottom,
    is increasing.
\end{enumerate}

An immaculate tableau is said to be \emph{standard} if it has content
$1^{|\alpha|}$.

Let $K_{\alpha, \beta}$ denote the number of immaculate tableaux of shape
$\alpha$ and content $\beta$.
\end{Definition}

We re-iterate that aside from the first column, there is no relation on the other
columns of an immaculate tableau. 

\begin{Example}\label{ex:immaculatetableau}
The five immaculate tableau of shape $[4,2,3]$ and content $[3,1,2,3]$: 
\[ \tikztableausmall{{1,1, 1, 3},{2, 3}, {4,4,4}} 
\tikztableausmall{{1,1, 1, 3},{2, 4}, {3,4,4}} 
\tikztableausmall{{1,1, 1, 4},{2,3}, {3,4,4}} 
\tikztableausmall{{1,1, 1, 4},{2, 4}, {3,3,4}} 
\tikztableausmall{{1,1, 1, 2},{3, 3}, {4,4,4}} 
\]
\end{Example}

\begin{Definition} \label{def:descentSIT}
We say that a standard immaculate tableau
$T$ has a descent in position $i$
if  $i+1$ is in a row strictly below the row containing $i$. The \emph{descent composition}, denoted $D(T)$, is the composition corresponding to the set of descents in $T$.
\end{Definition}

\begin{Example}
The standard immaculate tableau of shape $[6,5,7]:$
$$T = \tikztableausmall{{1,2,4,5, 10,11},{3, 6, 7, 8, 9}, {12,13,14,15,16, 17, 18}} $$
has descents in positions
    $\{2, 5, 11\}$.
    The descent composition of $T$ is then $D(T) = [2,3,6,7]$.
\end{Example}

\subsection{The immaculate basis of $\Nsym$}

The immaculate basis of $\Nsym$ was introduced in \cite{BBSSZ}. It shares many properties with the 
Schur basis of $\sym$. We define\footnote{This is not the original definition, 
but is equivalent by Proposition 3.16 in \cite{BBSSZ}.} the immaculate basis 
$\{ \fS_\alpha\}_{\alpha}$ as the unique elements of $\Nsym$ satisfying:
\[ \HH_\beta = \sum_{\alpha} K_{\alpha, \beta} \fS_\alpha.\]

\begin{Example}
Continuing from Example \ref{ex:immaculatetableau}, we see that \[ \HH_{3123} = \cdots + 5 \fS_{423} + \cdots.\]
\end{Example}

We will not attempt to summarize everything that is known about this basis, but instead 
refer the reader to \cite{BBSSZ} and \cite{BBSSZ2}.
% our previous \cite{BBSSZ} and upcoming \cite{BBSSZ2} papers.

\section{Modules for the dual immaculate basis}

In this section we will construct indecomposable modules for the $0$-Hecke algebra whose 
characteristic is a dual immaculate quasi-symmetric function.

\subsection{Quasi-symmetric functions}

The algebra of quasi-symmetric functions, $\Qsym$, was introduced in \cite{Ges} 
(see also subsequent
references such as \cite{GR, Sta84}). The  graded  component  $\Qsym_n$  is  indexed  by  compositions  of  $n$. 
The algebra is most readily realized as a subalgebra of the 
ring of power series of bounded degree 
$\mathbb{Q}[\![x_1, x_2, \dots]\!]$, and the monomial 
quasi-symmetric function indexed by a composition $\alpha$ is defined as
\begin{equation}
    \label{monomial-qsym}
    M_\alpha = \sum_{i_1 < i_2 < \cdots < i_m} x_{i_1}^{\alpha_1} x_{i_2}^{\alpha_2} \cdots x_{i_m}^{\alpha_m}.
\end{equation}
The algebra of quasi-symmetric functions, $\Qsym$, can be defined as the linear span of the monomial quasi-symmetric functions. These, in fact, form a basis of $\Qsym$, and their multiplication is inherited from $\mathbb{Q}[\![x_1, x_2, \dots]\!]$. 
We view $\sym$ as a subalgebra of $\Qsym$. In fact, the quasi-symmetric monomial functions
refine the usual monomial symmetric functions $m_\lambda \in \sym$:
\[ m_\lambda = \sum_{\sort(\alpha) = \lambda} M_\alpha,\]
where $\sort(\alpha)$ denotes the partition obtained by organizing the parts of $\alpha$ from the largest to the smallest.

The fundamental quasi-symmetric function, denoted $F_\alpha$,  is defined by its expansion
in the monomial quasi-symmetric basis: 
\[F_\alpha = \sum_{\beta \leq \alpha} M_\beta.\]

The algebras $\Qsym$ and $\Nsym$ form dual graded Hopf algebras. In this context, the monomial basis of $\Qsym$ 
is dual to the complete homogeneous basis of $\Nsym$. Duality can be expressed by the means of an inner product, for which $\langle H_\alpha,M_\beta \rangle = \delta_{\alpha,\beta}$.

In \cite{BBSSZ}, we studied the dual basis to the immaculate functions of $\Nsym$, denoted $\fS_\beta^*$ and indexed by compositions. They are the basis of $\Qsym$ defined by 
$\langle \fS_\alpha, \fS_\beta^*\rangle = \delta_{\alpha, \beta}$.
In \cite[Proposition 3.37]{BBSSZ}, we showed that the dual immaculate functions have the following positive expansion into the fundamental basis:

\begin{Proposition}\label{prop:FundamentalPositive}
The dual immaculate functions $\fS_\alpha^*$ are fundamental positive. Specifically they expand as 
\[ \fS_\alpha^* = \sum_{T}  F_{D(T)},\] a sum over all standard immaculate tableaux of shape $\alpha$.
\end{Proposition}

\subsection{Finite dimensional representation theory of $H_n(0)$}

We will outline the study of the finite dimensional representations of the $0$-Hecke algebra and its relationship to $\Qsym$. 
We begin by defining the $0$-Hecke algebra. We refer the reader to \cite[Section 5]{Th2} for the relationship between 
the generic Hecke algebra and the $0$-Hecke algebra and their connections to representation theory.

\begin{Definition} The Hecke algebra $H_n(0)$ is generated by the elements $\pi_1, \pi_2, \dots \pi_{n-1}$ subject to relations:
\begin{align*}
\pi_i^2 &= \pi_i;\\
 \pi_i \pi_{i+1}\pi_i &= \pi_{i+1}\pi_i \pi_{i+1};\\
\pi_i \pi_j &= \pi_j \pi_i \textrm{ if } |i-j| > 1.
\end{align*}
\end{Definition}
A basis of $H_n(0)$ is given by the elements $\{ \pi_\sigma : \sigma \in S_n \}$, 
where $\pi_\sigma = \pi_{i_1} \pi_{i_2} \cdots \pi_{i_m}$ if $\sigma = s_{i_1} s_{i_2}\cdots s_{i_m}$.

We let $G_0(H_n(0))$ denote the Grothendieck group of finite dimensional representations of $H_n(0)$. 
As a vector space, $G_0(H_n(0))$ is spanned by the finite dimensional representations of $H_n(0)$, 
with the relation on isomorphism classes $[B] = [A]+[C]$ whenever there is a short exact sequence of $H_n(0)$-representations 
$0\rightarrow A \rightarrow B \rightarrow C \rightarrow 0$.
We let \[ \mathcal{G} = \bigoplus_{n \geq 0} G_0(H_n(0)).\] 

The irreducible representations of $H_n(0)$ are indexed by compositions. 
The irreducible representation corresponding to the composition $\alpha$ is 
denoted $L_\alpha$. The collection $\{ [L_\alpha] \}$ 
forms a basis for $\mathcal{G}$. As shown in Norton \cite{N}, each irreducible 
representation is one dimensional, spanned by a non-zero vector $v_\alpha \in L_\alpha$, 
and is determined by the action of the generators on $v_\alpha$:
\begin{equation}
 \pi_i v_\alpha = \begin{cases} 
      0 & \textrm{ if $i \in \mathcal{S}(\alpha)$}; \\
      v_\alpha & \textrm{ otherwise}, \\
   \end{cases}
\end{equation}
where $\mathcal{S}(\alpha)$ denotes the subset of $[1 \dots n-1]$ corresponding to the composition $\alpha$.
The tensor  product $H_n(0) \otimes H_m(0)$ is naturally embedded as a subalgebra 
of $H_{n+m}(0)$. Under this identification, one can endow $\mathcal{G}$ with a ring structure; 
for $[N] \in G_0(H_n(0))$ and $[M] \in G_0(H_m(0))$, let \[ [N][M] := [Ind_{H_n(0) \otimes H_m(0)}^{H_{n+m}(0)} N \otimes M]\]
where induction is defined in the usual manner.

There is an important linear map $\mathcal{F}: \mathcal{G} \rightarrow \Qsym$ defined by $\mathcal{F}([L_\alpha])
 = F_\alpha$. For a module $M$, $\mathcal{F}([M])$ is called the \textit{characteristic of $M$}.
  
\begin{Theorem}[Duchamp, Krob, Leclerc, Thibon \cite{DKLT}]
The quasi-symmetric functions and the Grothendieck group of finite dimensional representations of 
$H_n(0)$ are isomorphic as rings. The map $\mathcal{F}$ is an isomorphism between $\mathcal{G}$ and $\Qsym$.
\end{Theorem}

\begin{Remark} 
The map $\mathcal{F}$ is actually an isomorphism of graded Hopf algebras. We will not make use of the coalgebra structure.
\end{Remark}

\subsection{A representation on $\mathcal{Y}$-words} We start by defining the analogue of a permutation module for $H_n(0)$. 
For a composition $\alpha = [\alpha_1, \alpha_2, \dots, \alpha_m] \models n$, we let $\mathcal{M}_\alpha$ denote
  the vector space spanned by words of length $n$ 
on $m$ letters with content $\alpha$ (so that $j$ appears $\alpha_j$ times in each word). 
The action of $H_n(0)$ on a word $w = w_1 w_2 \cdots w_n$ is defined on generators as:
\begin{equation}\label{eq:actionwords}
\pi_i w = \begin{cases} 
      w & \textrm{ if $w_i  \geq w_{i+1}$}; \\
      s_i(w) & \textrm{if $w_i < w_{i+1}$}; 
   \end{cases}
\end{equation}
where $s_i(w) = w_1 w_2 \cdots w_{i+1}w_i \cdots w_n$.
This is isomorphic to the representation: 
\[ Ind^{H_n(0)}_{H_{\alpha}(0)} 
\left(L_{\alpha_1} \otimes L_{\alpha_2} \otimes \cdots \otimes L_{\alpha_m}\right),\]
where $L_k$ is the one-dimensional representation indexed by the composition $[k]$ and
$H_\alpha(0) := H_{\alpha_1}(0) \otimes H_{\alpha_2}(0) \otimes \cdots \otimes H_{\alpha_m}(0)$. This can be seen by associating the element 
$\pi_v \otimes_{H_\alpha(0)} 
L_{\alpha_1} \otimes L_{\alpha_2} \otimes \cdots \otimes L_{\alpha_m}$
where $v$ is the minimal length left coset representative of $S_n / S_{\alpha_1} \times S_{\alpha_2} \times \cdots \times S_{\alpha_m}$
with the element $\pi_v (1^{\alpha_1} 2^{\alpha_2} \cdots k^{\alpha_k})$.

We call a word a \textit{$\mathcal{Y}$-word} if the first instance of $j$ appears before the first instance of $j+1$ for every $j$.
We let $\mathcal{N}_\alpha$ denote the subspace of $\mathcal{M}_\alpha$ consisting of all words that are not 
$\mathcal{Y}$-words. The action of $H_n(0)$ on $\mathcal{M}_\alpha$ will never move a $j+1$ to the right of a $j$. 
This implies that $\mathcal{N}_\alpha$ is a submodule of $\mathcal{M}_\alpha$. 
The object of our interest is the quotient module $\mathcal{V}_\alpha := \mathcal{M}_\alpha / \mathcal{N}_\alpha$. We now state our main result.

\begin{Theorem}\label{thm:repthry}
The characteristic of $\mathcal{V}_\alpha$ is the dual immaculate function indexed by 
$\alpha$, i.e. $\mathcal{F}([\mathcal{V}_\alpha]) = \fS^*_\alpha$.
\end{Theorem}

Before we prove this we will associate words to standard immaculate tableaux and give an equivalent description 
of the $0$-Hecke algebra on standard immaculate tableau. To a $\mathcal{Y}$-word $w$, we associate the unique 
standard immaculate tableau $\mathcal{T}(w)$ which has a $j$ in row $w_j$.

\begin{Example}\label{bijectionYwords}
Let $w = 112322231$ be the $\mathcal{Y}$-word of content $[3,4, 2]$. Then $\mathcal{T}(w)$ is the standard immaculate tableau:
\[ \tikztableausmall{{1,2,9},{3,5,6,7},{4,8}} \]
\end{Example}

\begin{Remark} $\mathcal{T}$
 yields a bijection between standard immaculate tableau and $\mathcal{Y}$-words.
\end{Remark}

\begin{Remark}
In the case of the symmetric group, the irreducible representation corresponding to the partition 
$\lambda$ has a basis indexed by standard tableaux. Under the same map $\mathcal{T}$, standard Young 
tableaux are in bijection with Yamanouchi words (words for which every prefix contains at least as many 
$j$ as $j+1$ for every $j$). In this sense, $\mathcal{Y}$-words are a natural analogue to 
Yamanouchi words in our setting. The Specht modules that give rise to the indecomposable module of the symmetric group 
are built as a quotient of  ${\mathcal M}_\lambda$. Under the Frobenius map, these modules are associated to Schur functions. 
\end{Remark}

We may now describe the action of $H_n(0)$ on $\mathcal{V}_\alpha$, identifying the set of standard immaculate tableaux as the basis. 
Specifically, for a tableau $T$ and a generator $\pi_i$, we let:

\begin{equation}\label{eq:actiontableaux}
 \pi_i (T) = \begin{cases} 
      0 & \textrm{ if $i$ and $i\!+\!1$ are in the first column of $T$} \\
      T & \textrm{ if $i$ is in a row weakly below the row containing $i\!+\!1$}\\
      s_i(T) & \textrm{ otherwise};
   \end{cases}
\end{equation}
where $s_i(T)$ is the tableau that differs from $T$ by swapping the letters $i$ and $i+1$.

\begin{Example}
Continuing from Example \ref{bijectionYwords}, we see that $\pi_1, \pi_4, \pi_5, \pi_6, \pi_8$ send $T$ to itself, 
$\pi_3$ sends $T$ to $0$ and $\pi_2, \pi_7$ send $T$ to the following tableaux:
\[ \pi_2(T) = \tikztableausmall{{1,3,9},{2,5,6,7},{4,8}} \hspace{1in} \pi_7(T) = \tikztableausmall{{1,2,9},{3,5,6,8},{4,7}} \]
\end{Example}

An example of the full action of $\pi_i$ on tableaux representing the basis elements of the module $\mathcal{V}_{(2,2,3)}$
is given in Figure \ref{fig:module223}.  If we order the tableaux so that $S \prec T$ 
if there exists a permutation $\sigma$ such that $\pi_\sigma(T) = S$ then this figure shows that order is not a total order on
tableaux but that it can be extended to a total order arbitrarily.  We will use this total order in the following proof of
Theorem \ref{thm:repthry}.

We are now ready to prove Theorem \ref{thm:repthry}, which states that the characteristic of $\mathcal{V}_\alpha$ is $\fS_\alpha^*$.

\begin{proof}[Proof of Theorem \ref{thm:repthry}]
We construct a filtration of the module $\mathcal{V}_\alpha$ whose successive quotients are irreducible representations. 
Now, define $\mathcal{M}_T$ to be the linear span of all standard immaculate tableaux 
that are less than or equal to $T$. From the definition of the order and the fact that the $\pi_i$ 
are not invertible, we see that $\mathcal{M}_T$ is a module. Ordering the standard immaculate 
tableaux of shape $\alpha$ as $T_1, T_2, \dots, T_m$, then we have a filtration of $\mathcal{V}_\alpha$:
\[ 0 \subset \mathcal{M}_{T_1} \subset \mathcal{M}_{T_2} \subset \cdots \subset \mathcal{M}_{T_m} = \mathcal{V}_\alpha.\] 
The successive quotient modules $\mathcal{M}_{T_j}/\mathcal{M}_{T_{j-1}}$ are one dimensional, spanned by $T_j$; 
to determine which irreducible this is, it suffices to compute the action of the generators. From the description of 
$\mathcal{V}_\alpha$ above, we see that 
\begin{equation} \pi_i (T_j) = \begin{cases} 
      0 & \textrm{ if $i \in \mathcal{S}(D(T_j))$}\\
      T_j & \textrm{ otherwise}.
   \end{cases}
\end{equation}
This is the representation $[L_{D(T_j)}]$, whose characteristic is $F_{D(T_j)}$. 
Therefore $\mathcal{F}([\mathcal{V}_\alpha]) = \fS_\alpha^*$ by Proposition \ref{prop:FundamentalPositive}.
\end{proof}

We aim to prove that the modules we have constructed are indecomposable. 
We let $\SS_\alpha$ denote the super-standard tableau of shape $\alpha$, namely, the unique standard immaculate tableau with the first $\alpha_1$ letters in the first row, 
the next $\alpha_2$ letters in the second row, etc.  We will first need a few lemmas.

\todo{LS: I don't really like $S_\alpha$ for super standard tableau, we already have $\mathcal{S}$ for descent sets and $S_n$ for the symmetric group. I changed it to $SS$ which I also don't really like, but made a macro in case anyone can think of a better notation.

CB: Mathfrak to the rescue! I figured $\mathfrak{v}$ should generate $\mathcal{V}$. Objections?
}

\todo{MZ: The following proof is new.  It might need work to make as clear as possible. \\ LS: Wrote it in terms of Y-words instead (and added Remark 3.9)... not fully happy about what I wrote on the action commuting with the bijection, if anyone wants to take a look.
Commented Mike's proof in case anyone wants to go back to writing it in terms of tableaux.
\\
CB: Added my own proof, which I think is very simple. Left the others as comments.
}

\begin{Lemma}\label{lemma:cycgen}
The module $\mathcal{V}_\alpha$ is cyclicly generated by $\SS_\alpha$.
\end{Lemma}

\begin{proof}
The module $\mathcal{M}_\alpha$ is cyclically generated by $1^{\alpha_1} 2^{\alpha_2} \cdots k^{\alpha_k} = \mathcal{T}^{-1}(\SS_\alpha)$, which can be seen since every basis element of $\mathcal{M}_\alpha$ comes from an application of the anti-sorting operators $\pi_i$ on $1^{\alpha_1} 2^{\alpha_2} \cdots k^{\alpha_k}$.

$\mathcal{V}_\alpha$ is a quotient of $\mathcal{M}_\alpha$, and hence cyclicly generated by the same element.
\end{proof}

\begin{Lemma}\label{lemma:turkey}
If $P$ is a standard immaculate tableau of shape $\alpha$ such that $\pi_i(P) = P$ 
for all $i \in \{1,2,\cdots,n\} \setminus \mathcal{S}(\alpha)$
%for all $i \in \{1,2,\dots, \alpha_1-1\} \cup \{\alpha_1+1, \dots, \alpha_1+\alpha_2-1 \} \cup \cdots 
%\cup \{\alpha_1+\alpha_2+\dots + \alpha_{m-1}+1, \dots, n-1\}$ 
 then $P = \SS_\alpha$. 
In particular, if $P \neq \SS_\alpha$ then there exists an $i$ such that $\pi_i(\SS_\alpha) = \SS_\alpha$ but $\pi_i(P) \neq P$. 
\end{Lemma}

\begin{proof}
If $\pi_i(P) = P$, then $i$ must be in the cell to the left of $i+1$ or in a row below $i+1$. 
The fact that $\pi_i(P) = P$ for all $i \in \{ 1,2,\dots,\alpha_1-1\}$ implies that the first row of $P$ agrees with $\SS_\alpha$. 
In a similar manner, we see that the second rows must agree. 
Continuing in this manner, we conclude that $P  = \SS_\alpha$.
\end{proof}

\begin{Theorem}
For every $\alpha \models n$, $\mathcal{V}_\alpha$ is an indecomposable representation of $H_n(0)$.
\end{Theorem}

\begin{proof}
We let $f$ be an idempotent module morphism from $\mathcal{V}_\alpha$ to itself. 
If we can prove $f$ is either the zero morphism or the identity, then $\mathcal{V}_\alpha$ is indecomposable \cite[Proposition 3.1]{Ja}. 

Suppose $f(\SS_\alpha) = \sum_T a_T T$. 
By Lemma \ref{lemma:turkey}, for any $P \neq \SS_\alpha$, there exists an $i$ such that $\pi_i (\SS_\alpha) = \SS_\alpha$ but $\pi_i(P) \neq P$. 
Since $f$ is a module map, 
\begin{equation}\label{eqn:boring}\sum_T a_T T = f(\SS_\alpha) = f(\pi_i\SS_\alpha) = \pi_if(\SS_\alpha) = \sum_T a_T \pi_i T.\end{equation} 

The coefficient of $P$ on the right-hand side of Equation \eqref{eqn:boring} is zero 
(if there was a $T$ such that $\pi_i T = P$ then $\pi_i T = \pi_i^2T  = \pi_i P \neq P$, a contradiction). 
Therefore $a_P = 0$ for all $P\neq \SS_\alpha$, so $f(\SS_\alpha) = \SS_\alpha$, or $f(\SS_\alpha) = 0$. 
Since $\mathcal{V}_\alpha$ is cyclicly generated by $\SS_\alpha$, this implies that either $f$ is the identity morphism or the zero morphism.
\end{proof}

\begin{figure}[h]\label{fig:module223}
\begin{center}
\includegraphics[width=3in]{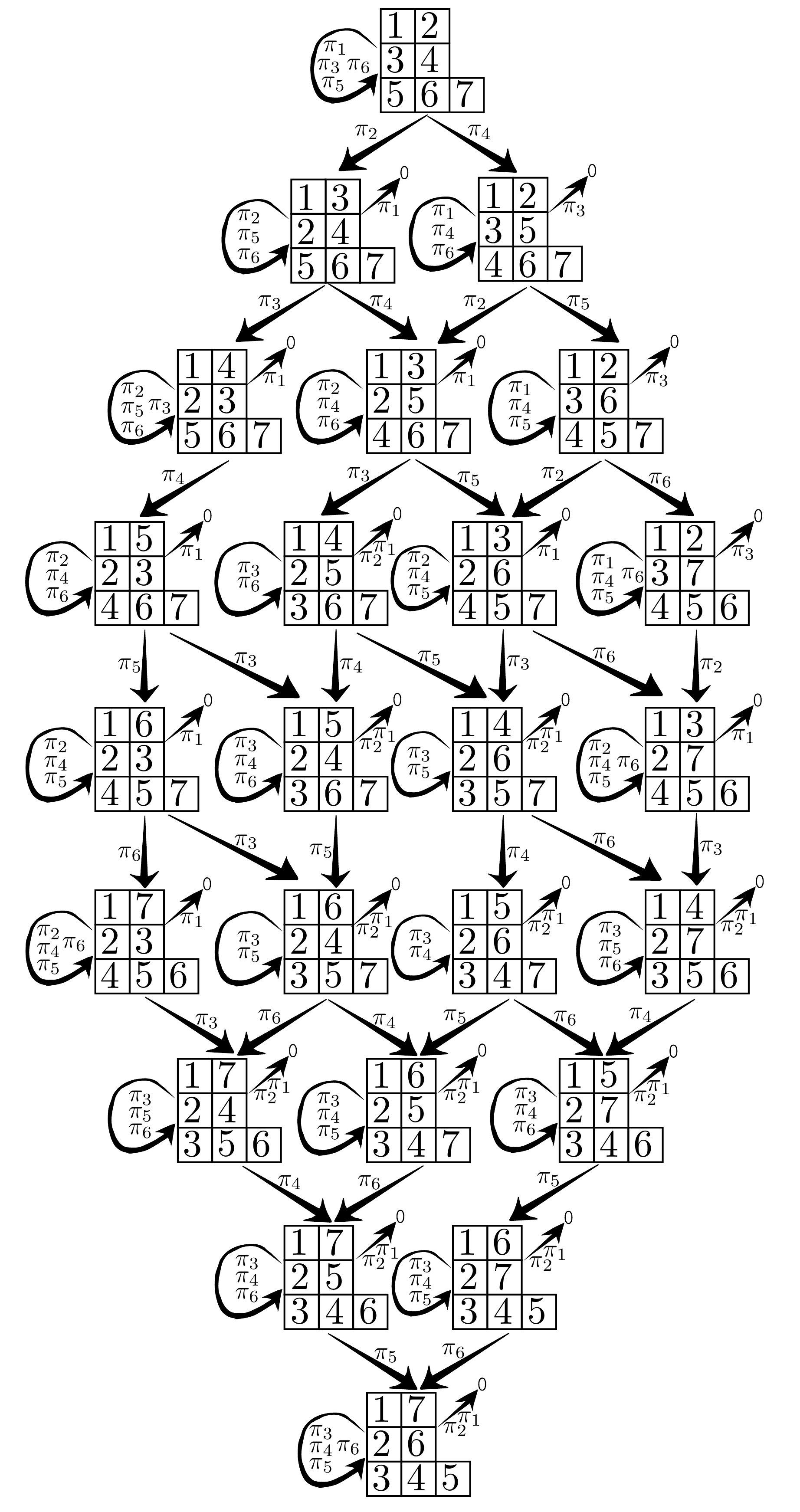}
\end{center}
\caption{A diagram representing the action of the generators $\pi_i$ of $H_n(0)$ given in Equation \eqref{eq:actiontableaux}
on the basis elements of the module $\mathcal{V}_{(2,2,3)}$.}
\end{figure}

\end{document}